\documentclass[nohyperref]{article}

\usepackage{microtype}
\usepackage{graphicx}
\usepackage{booktabs} 

\usepackage{hyperref}

\usepackage[accepted]{icml2022}

\usepackage{amsmath}
\usepackage{amssymb}
\usepackage{mathtools}
\usepackage{amsthm}

\usepackage[capitalize,noabbrev]{cleveref}

\theoremstyle{plain}
\newtheorem{theorem}{Theorem}[section]

\theoremstyle{definition}
\newtheorem{definition}[theorem]{Definition}
\newtheorem{assumption}[theorem]{Assumption}
\theoremstyle{remark}
\newtheorem{remark}[theorem]{Remark}

\usepackage[textsize=tiny]{todonotes}

\usepackage{makecell}

\usepackage{subcaption}

\def\##1\#{\begin{align}#1\end{align}}
\def\$#1\${\begin{align*}#1\end{align*}}

\def\eps{\varepsilon}

\usepackage{xcolor}

\newcommand{\ignore}[1]{}

\newcommand{\bfm}[1]{\ensuremath{\mathbf{#1}}}

\DeclarePairedDelimiterX{\infdivx}[2]{(}{)}{%
  #1\;\delimsize\|\;#2%
}

\usepackage{titlesec}
\titlespacing\section{0pt}{2pt plus 1pt minus 1pt}{1pt plus 1pt minus 1pt}
\titlespacing\subsection{0pt}{2pt plus 1pt minus 1pt}{1pt plus 1pt minus 1pt}

\def\bg{\bfm g}

     \def\NN{\mathbb{N}}

     \def\RR{\mathbb{R}}

\def\calA{{\cal  A}} \def\cA{{\cal  A}}
 
 \def\cC{{\cal  C}}

 \def\cH{{\cal  H}}

\def\calK{{\cal  K}} 
 \def\cL{{\cal  L}}
 \def\cM{{\cal  M}}

\newcommand{\bfsym}[1]{\ensuremath{\boldsymbol{#1}}}

 \def\bzeta{\bfsym {\zeta}}

\DeclareMathOperator*{\argmin}{arg\,min}

\DeclareMathOperator{\regret}{\mathcal{R}}

\DeclarePairedDelimiterX{\ip}[2]{\langle}{\rangle}{#1, #2}

\usepackage{cool}
\newcommand{\eqdef}{\stackrel{\text{def}}{=}}

\usepackage{amsthm}

\icmltitlerunning{A Regret Minimization Approach to Multi-Agent Control}

\begin{document}

\twocolumn[
\icmltitle{A Regret Minimization Approach to Multi-Agent Control}



\icmlsetsymbol{equal}{*}

\begin{icmlauthorlist}
\icmlauthor{Udaya Ghai}{pu,goog}
\icmlauthor{Udari Madhushani}{pu2}
\icmlauthor{Naomi Leonard}{pu2}
\icmlauthor{Elad Hazan}{pu,goog}

\end{icmlauthorlist}

\icmlaffiliation{pu}{Department of Computer Science, Princeton University, Princeton, NJ}
\icmlaffiliation{pu2}{Department of Mechanical and Aerospace Engineering, Princeton University, Princeton, NJ}
\icmlaffiliation{goog}{Google AI Princeton, Princeton, NJ}

\icmlcorrespondingauthor{Udaya Ghai}{ughai@cs.princeton.edu}

\icmlkeywords{Machine Learning}

\vskip 0.3in
]



\printAffiliationsAndNotice{}  

\begin{abstract}
We study the problem of  multi-agent control of a dynamical system with  known dynamics and adversarial disturbances. Our study focuses on optimal control without centralized precomputed policies, but rather with adaptive control policies for the different agents that are only equipped with a stabilizing controller. We give a reduction from any (standard) regret minimizing control method to a distributed algorithm. The reduction guarantees that the resulting distributed algorithm has low regret relative to the optimal precomputed joint policy. Our methodology involves generalizing online convex optimization to a multi-agent setting and applying recent tools from nonstochastic control derived for a single agent. We empirically evaluate our method on a model of an overactuated aircraft. We show that the distributed method is robust to failure and to adversarial perturbations in the dynamics. 

\end{abstract}

\vspace{-5mm}\section{Introduction}
In many real-world scenarios it is desirable to apply multi-agent control
algorithms, as opposed to a centralized designed policy, for a number of  practical
reasons. First is increased robustness:  agents, serving as control actuators, may be added or removed
from an existing set, without change to the distributed policies. This
also allows for fault-tolerance; agents can adapt  on the fly to faulty actuators.
Second, limited computational resources and/or system information may limit
applicability of a sophisticated centralized policy, whereas simple distributed policy
agents are feasible and can adapt to the underlying dynamics.  Third,
multi-agent control allows robustness to noisy state and control observations
by the individual agents. These advantages motivate the study of
distributed control, and indeed a vast literature exists on the design and analysis of such methods. 

In this paper we propose a novel game-theoretic performance metric for multi-agent control (MAC), which we call {\bf multi-agent regret}. This metric measures the difference between the total loss of the joint policy of individual agents vs. the loss of the best joint policy in hindsight from a certain reference class of policies. Vanishing multi-agent regret thus implies competitive performance with respect to that of the optimal policy in a certain policy class. 


We study a mechanism for MAC that takes individual controllers and merge them into a MAC method with vanishing regret. 
For this purpose, we make use of recent advances in online learning for control and specifically the framework of {\it nonstochastic control}.
This methodology was recently proposed for the study of robust adaptive control algorithms through the lens of online convex optimization. As opposed to optimal and robust control, nonstochastic control studies adaptive control algorithms that minimize {\it regret}, or the difference in loss/reward of the controller vs. the best policy in hindsight from a reference class.  This is a game-theoretic performance metric, which is naturally applicable to agents interacting in a multi-agent environment. 


\subsection{Our result and techniques}

Our main result is a reduction from single agent to multi-agent control with provable regret guarantees. More precisely, we assume that we have access to single-agent controllers with the following guarantees:
\vspace{-2mm}
\begin{enumerate}
\item Each agent has sublinear regret vs. a class of policies under adversarial perturbations and cost functions.\vspace{-2mm}
\item Each agent should be able to evaluate their policy given the controls of the other agents. \vspace{-2mm}
\item The policy that each agent uses is independent of the controls that the other agents play. \vspace{-2mm}
\item The joint cost function over all agents needs to be jointly convex. \vspace{-2mm}
\end{enumerate}
This set of assumptions is well motivated: we review several well-studied control settings in the literature which provide  agents with such guarantees. Under these assumptions we give a reduction that takes regret minimizing control agents and guarantees low multi-agent regret of the joint control policy, w.r.t. the optimal joint policy in hindsight.

Let $\cA$ be a single agent control algorithm\footnote{More formally, it is an online convex optimization with memory algorithm, which we precisely define in coming sections.} that, given a state $x_t$, produces a control $u_t$, and suffers cost $c_t(x_t,u_t)$. The average (policy) regret of $\cA$ w.r.t. policy class $\Pi$ is defined to be the difference between its total cost and that of the best policy in hindsight, namely 
$$\regret_T(\cA) = \frac{1}{T} \sum_t c_t(x_t,u_t) - \min_{\pi \in \Pi} \frac{1}{T} \sum_t c_t(x_t^\pi,u_t^\pi) ~. $$
Notice that the comparator are the states $x_t^\pi$ and controls $u_t^\pi$ that would have been played by policy $\pi$. Then algorithm $\cC$, described in Algorithm box~\ref{algorithm:multiagent_ctrl_generic}, takes as an input $k$ such controllers, and under the aforementioned assumptions, controls a multi-agent system with the guarantee that 
\vspace{-2mm}
\begin{align*}
    \regret_T(\cC) \le  \sum_{i =1}^k \regret_T(\cA_i) + \tilde{O} \left( \frac{1}{T} + \eps \right) ~,
    \vspace{-2mm}
\end{align*}
where $\cA$ is the control algorithm for agent $i$, which is not necessarily the same for every agent, and $\eps$ is an upper bound on the error of each agent's policy evaluation.
Here $\regret_T(\cC)$ is the multi-agent regret of the joint policy, namely the difference between the cost of $\cC$, and that of the best joint-policy in hindsight. The reference policy class we use for comparison is the Cartesian product of the individual agent policy classes.  This is defined precisely in Section \ref{sec:control_v2}.

In order to obtain this result we study a generalization of online convex optimization to the multi-agent setting. Specifically, we propose a reduction from online convex optimization (OCO) to multi-agent OCO, such that the overall multi-agent regret is guaranteed to be the sum of the individual agents' regret on an induced loss function.

\subsection{Paper structure}

Section \ref{sec:setting} describes the formal setting in which we operate, as well as give basic definitions and assumptions. Section \ref{sec:multioco} provides the foundation of our result: a multi-agent generalization of online convex optimization, definition of multi-agent regret, and an efficient reduction from OCO to multi-agent OCO. In section \ref{sec:control_v2} we give our main result: a meta-algorithm that converts regret minimizing controllers into a joint multi-agent regret minimizing distributed policy. Experiments are shown in section \ref{sec:exp}. Some analysis is deferred to the Appendix. 

\subsection{Related work}

{\bf{Multi-agent RL.}}
Multi-agent RL \cite{sutton2018reinforcement} considers the problem of finding globally optimal joint policies through optimizing local policies. Common approaches for finding optimal policies in RL includes dynamic programming, Monte-Carlo methods \cite{williams1992simple}, temporal difference learning \cite{watkins1992q}, combining temporal difference and Monte-Carlo learning, evaluation strategies \cite{salimans2017evolution}, and policy gradient methods \cite{williams1992simple, mnih2016asynchronous}. Policy gradient methods, which are closest to our algorithm, compute stochastic gradients from trajectories \cite{sutton2000policy,sutton2018reinforcement}. More recently, \citet{jin2021v} provide a decentralized value-based method with game-theoretic equilibrium guarantees. Performance in multi-agent coordination significantly improves with agents' ability to predict the behaviour of other agents. In opponent modeling methods, wherein agents learn to model the behaviours of their partners, \cite{foerster2018learning} agents use counterfactual policies of other agents to update their policy parameters. In contrast, our approach decouples the agents' policies, forgoing the need to deal with counterfactuals of other agents. 

{\bf{Decentralized and Distributed control.}}
Communication-free decentralized control is the earliest multi-agent control model considered in the literature. In this model, each agent observes a different partial observation of the state and each agent provides its own control input to the system. This classic setting is well understood with exact characterizations of stabilizability \cite{wang1973} and controllability \cite{lefebvre1980}. More recently, decentralized control, also called distributed control, has been generalized to include a network layer of communication among the controllers, modeled by a network graph. Coordinated and cooperative control over a communication network has been studied using a wide range of approaches, such as distributed optimization \cite{Nedic2018}, model predictive control \cite{christofides2013distributed,sturz2020distributed}, and for a wide variety of problems, including stabilization of formations, coverage, and search \cite{Olfati2007,Bullo2009,Cao2013,Knorn2016}. 

One particular direction in this setting is on designing distributed observers, using information from neighbors to provide local estimates of the global state \cite{Olfati2005,ruggero2008DistributedKalman, matei2010consensusfilter, das2017consensusinnovations, park2017distlti}.  Such distributed observers can be applied to produce end-to-end control laws. One such approach stabilizes systems when the neighbor graph is strongly connected and the joint system is stabilizable \cite{wang2020distcontrol}.  In contrast to the present work, approaches to this problem are decentralized from an \emph{information} perspective, but not a \emph{game theoretic} perspective as controllers are designed centrally.

{\bf{Exploiting input redundancy in control.}}  Input redundancy in overactuated systems enables efficient responses \cite{sobelEigenstructureAssignment} and robustness to failures \cite{oppenheimer2006, Tohidi2017AdaptiveCA}. A number of optimization methods have been developed for redistributing desired control input under different faulty conditions. Our approach does not tackle this directly, but the action-centric rather than policy-centric philosophy of the algorithm leads to some robustness in such settings. 

{\bf{Online learning and online convex optimization.}} The framework of learning in games has been extensively studied as a model for learning in adversarial and nonstochastic environments  \cite{cesa2006prediction}. 
Online learning was infused with algorithmic techniques from mathematical optimization into the setting of  online convex optimization, see \cite{hazan2019introduction} for a comprehensive  introduction. We particularly make use of the property that no-regret algorithms converge to equilibrium, a phenomenon originally proposed by \cite{hart2000simple}.

{\bf{Online and nonstochastic control.}}
The starting point of our study are algorithms which enjoy sublinear regret for online control of dynamical systems; that is, whose performance tracks a given benchmark of policies up to a term which is vanishing relative to the problem horizon.  \cite{abbasi2011regret} initiated the study of online control under the regret benchmark for Linear Time-invariant (LTI) dynamical systems.
Our work instead adopts the \emph{nonstochastic control setting} \cite{agarwal2019online}, that allows for adversarially chosen (i.e. non-Gaussian and potentially adaptively chosen \cite{ghaiAdvDisturbance}) noise and costs that may vary with time.   The nonstochastic control model was extended to consider nonlinear and time-varying dynamics in  \cite{gradu2020adaptive,minasyan2021online,chen2021provable}.
See \cite{IcmlTutorial21} for a comprehensive survey of results and advances in online and nonstochastic control. 
\section{Problem Setting}\label{sec:setting}
\subsection{Notation}
The norm $\|\cdot \|$ refers to the $\ell_2$ norm for vectors and
spectral norm for matrices. For any natural number $n$, the set $[n]$ refers to the set $\{1, 2 \dots n\}$. 
We use the window notation $z_{a:b} \eqdef [z_a, \dots, z_{b}]$ to represent a sequence of
vectors. We use the convention where indices $i,j$ refer to agents, with superscripts
(typically) denoting an agent-specific quantity, while lack of a superscript denotes something global. 
We use the index $-i$ to represent information for all agents except $i$, for example $u^{-i}_t = (u^1_t \dots u^{i-1}_t, u^{i+1}_t \dots u^k_t)$.
\subsection{Multi-Agent Nonstochastic Control Problem}\label{sec:prob_setting}
We consider the following multi-agent control problem. Let $f_t: \RR^{d_x} \times \RR^{d_u} \rightarrow \RR^{d_x}$ and $g_i: \RR^{d_x} \rightarrow \RR^{d_{y_i}}$ define known time varying dynamics. Now, starting from state $x_0$, the system follows the dynamical equation 
\begin{align}
   y^i_{t} = g_i(x_{t}) + e^{i}_t, \label{eq:gen_dynamics} \quad x_{t+1} = f_t(x_t, u_t) + w_t,
\end{align}
where $x_t$ is the state, $u_t$ is the joint control consisting of $u^{i}_t \in \RR^{d_{u_i}}$ for each $i \in [k]$, $y^i_t$ is a partial observation, and $e^{i}_t, w_t$ are additive adversarial chosen disturbance terms. After committing to a control $u^i_t$, agent $i$ observes the controls of all other players $u^{-i}_t$ along with a convex cost $c_t: \RR^{d_x} \times \RR^{d_u} \rightarrow \RR$. Together, the $k$ agents suffer cost $c_t(x_t, u_t)$. The information model described above is diagrammed for a $2$-agent system in Figure.~\ref{fig:multiagent_system}. 

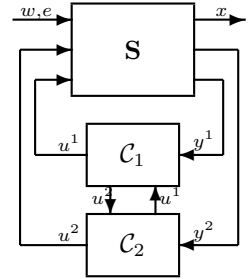
\begin{figure}[h]
\centering
\vskip7mm
\setlength{\unitlength}{0.1cm}
\begin{center}\begin{picture}(54,29)(0,31)
\thicklines
\put(24,54){\framebox(16,12){$\mathbf S$}}

\put(26,42){\framebox(12,8){$\cC_1 $}}
\put(26,30){\framebox(12,8){$\cC_2 $}}

\put(35, 38){\vector(0,1){4}}
\put(29, 42){\vector(0,-1){4}}
\put(35.5,39.5){$\scriptstyle u^1$}
\put(26.5,39.5){$ \scriptstyle u^2$}

\put(17,64.5){$\scriptstyle w, e$}
\put(16,64){\vector(1,0){8}}

\put(43,64.5){$\scriptstyle x$}
\put(40,64){\vector(1,0){7}}

\put(40,47){$\scriptstyle y^1$}
\put(44,56){\line(0,-1){10}}
\put(40,56){\line(1,0){4}}
\put(44,46){\vector(-1,0){6}}

\put(40,35){$\scriptstyle y^2$}
\put(40,60){\line(1,0){6}}
\put(46,60){\line(0,-1){26}}
\put(46,34){\vector(-1,0){8}}

\put(22,46.5){$\scriptstyle u^1$}
\put(26,46){\line(-1,0){7}}
\put(19,56){\vector(1,0){5}}
\put(19,46){\line(0,1){10}}

\put(22,34.5){$\scriptstyle u^2$}
\put(26,34){\line(-1,0){9}}
\put(17,34){\line(0,1){26}}
\put(17,60){\vector(1,0){7}}

\end{picture}\end{center}
\caption{Diagram of a $2$-agent System $\mathbf S$ executing dynamics from \eqref{eq:gen_dynamics} with agents $\cC_1, \cC_2$.}
\label{fig:multiagent_system}
\end{figure}
It should be noted that while each agent shares its action with other agents, it does not directly share its policy or its observation. Policy parameters might not be shareable due to memory/hardware restrictions (e.g. the policy is a neural net on a GPU), while lower dimensional controls may be less of a problem. Furthermore, some settings may involve local observations with private information, but with controls that are inherently public. It also is useful for seamlessly handling system failures as discussed in Section~\ref{sec:sys_fail}.

A control policy for agent $i$, 
$\pi^i:\bigcup_{s \in \NN} \RR^{s {d_{y_i}}} \times \RR^{(s-1) d_u} \rightarrow \RR^{d_u}$ is a mapping from all available information to a control $u^i_t = \pi^i(y_{1:t}, u_{1:{t-1}})$. A joint policy for the multi-agent setting is a collection of policies for each agent $\pi = \pi^{1:k}$. Policy classes are defined analogously with $\Pi^i$ being a set of agent $i$ policies and $\Pi$ being a set of joint policies. We often consider policies parameterized by vectors $\theta^i \subseteq \Theta_i$, where $\Theta_i$ is some bounded convex domain and use the parameterizations and policies interchangeably.

For time horizon $T$, performance of joint policy $\pi$ is measured by the average cost along the trajectory $(x^{\pi}_1, u^{\pi}_1 \dots)$,
$$J(\pi| w_{1:T}, e_{1:T}) = \frac{1}{T} \sum_{t=1}^T c_t(x^{\pi}_t, u^{\pi}_t)~.$$
 For a set of online control algorithms for each agent $\cC = \cC^{1:k}$, we are interested in the average (policy) \emph{regret} relative to some joint policy class $\Pi$, defined as 
\begin{align*}
  \regret_T(\cC) = J(\cC|w_{1:T}, e_{1:T}) - \inf_{\pi \in \Pi} J(\pi| w_{1:T}, e_{1:T})~.
\end{align*}
\vspace{-4mm}\subsection{Assumptions}
Converging towards optimal play for an online control algorithm may not happen immediately, so our system cannot explode too quickly. In order to bound the cost during this stage, $x_t$ must not explode in the absence of control. The corresponding outputs in the absence of any controller is a concept called \emph{Nature's $y$'s} in \cite{simchowitz2020improper}, which are an integral part of DAC policies, defined below. We analogously define \emph{Nature's $x$'s} here. 

\begin{definition}[Nature's $x$'s and $y$'s]
Given a sequence of disturbances $(w_t, e_t)_{t> 1}$ we define Nature's $x$'s and Nature's $y$'s respectively as the sequences $x^{\text{nat},i}_t, y^{\text{nat},i}_t$ where 
\begin{align*}
  y^{\text{nat}, i}_{t} = g_i(x_{t}) + e^{i}_t, \quad x^{\text{nat}}_{t+1} = f_t(x_t, 0) + w_t~.
\end{align*}
\end{definition}

\begin{assumption}[Bounded Nature's $x$'s]\label{assumption:bounded_nat_y}
We assume that $w_t$ and $e_t$ are chosen by an oblivious adversary,
and that $\|x_t^{\text{nat}}\| \le R_{\text{nat}}$ for all $t$.
\end{assumption}

We also need well behaved cost functions for learning:

\begin{assumption}\label{assumption:cost}
The costs $c_t(x, u)$ are convex and if $\|x\|, \|u\| \le D$ then there exists constant $C$
such that $0 \le c_t(x,u)) \le C D^2$~.
\end{assumption}

In control, the most ubiquitous cost functions are quadratics, which satisfy this assumption. Finally, we must restrict our joint policy class $\Pi$ in order to obtain global regret guarantees. We make the following assumption about each agent's policy class:

\begin{assumption}[{\bf Decoupling}] \label{assumption:dist_policy}
For any policy $\pi^i \in \Pi^i$, the policy can always be represented as a function of the disturbances $e^i_t, w_t$. In particular, there exists a function $\bar{\pi}:\bigcup_{s \in \NN} \RR^{(s-1) d_x} \times \RR^{s d_{y_i}} \rightarrow \RR^{d_u}$ for any set of controls $u_{1:t-1}$ and the resultant set of observations $y_{1:t}$, such that 
$\pi(y_{1:t}, u_{1:t-1}) = \bar{\pi}(w_{1:t-1}, e^{i}_{1:t})~.$
\end{assumption}

Note, that if Assumption~\ref{assumption:dist_policy} holds for each agent's policy class, then the joint policy class $\Pi$ will depend implicitly on only $w_{1:{t-1}}, e_{1:t}$. Policies that satisfy this assumption are special in that the controls played by one agent's policy are completely independent of another agent's policy, given the disturbances are chosen obliviously.

\subsection{Policy Classes}
We now introduce a few important policy classes that satisfy Assumption~\ref{assumption:dist_policy}.

{\bf{Disturbance-action Control.}}
For fully observed linear systems, disturbances $w_t$ can be derived from states and joint controls. This allows us to use Disturbance Action Control, a policy parameterization that is linear in the disturbance history.
\begin{definition}\label{DAC:def}
A Disturbance-action Controller (DAC) \cite{agarwal2019online}, $\pi(M)$, is specified by parameters $M \in \RR^{d_{u_i} \times m d_x}$ where the control $u^i_t = M w_{t-m:t-1}$.
\end{definition}
A comparator policy class $\Pi^i$ can be defined as a bounded convex set $\cM \subset \RR^{d_u \times m d_x}$ of policy matrices. It has been shown that DACs can approximate stable Linear Dynamic Controllers (LDCs), a powerful generalization of linear controllers that encompasses $\cH_2$ and $\cH_{\infty}$ optimal controllers under partial observation, and is a mainstay in control applications (see Definition~\ref{LDC:def} in Appendix~\ref{sec:LDC}). Furthermore, when costs are quadratic, it can be shown that a DAC also approximates the offline optimal policy \cite{goel2021competitive}.

{\bf{Disturbance-response Control.}}
Another policy class that suits our needs, especially if the system is partial observed, is that of Disturbance-Response Control (DRC). Unlike in the fully observed setting, with partial observations it is generally impossible for an agent to calculate $w_t$ for use in a policy. Instead, DRC controllers are a linear function of a window of previous Nature's $y$'s defined as follows:

\begin{definition}\label{DRC:def}
A Disturbance-response Controller (DRC) \cite{simchowitz2020improper}, $\pi(M)$, is specified by parameters $M \in \RR^{d_u \times m d_y}$ where the control $u^i_t = M y^{\text{nat}}_{t-m:t-1}$.
\end{definition}

Like DAC, DRC is a powerful control law which also can approximate stable LDCs. Furthermore, Nature's $y$'s are just a function of $w$'s and $e$'s so DRC policy classes satisfy Assumption~\ref{assumption:dist_policy}. 

\vspace{-3mm}\paragraph{Neural Network Policies}
We also consider neural network policies. In particular, fully-connected neural-networks with ReLU activations that act on normalized windows of $w$'s or Nature's $y$' were explored in \cite{chen2021provable}. While there are more restrictions for this class, neural-networks are a richer policy class than the linear counterparts DAC and DRC.

\section{Multiagent Online Convex Optimization With and Without Memory}\label{sec:multioco}

Our derivation henceforth for the  multi-agent control problem hinges upon a generalization of OCO to the  multi-agent setting. Now we describe this generalization, explain the challenge, and give a reduction from vanilla OCO, as well as OCO with Memory (OCO-M), to the  multi-agent settings.

In multiplayer OCO we have $k$ players and the $i$th learner iteratively 
chooses a vector from a convex set $\calK_i \subseteq \RR^d$. 
We denote the total number of rounds as $T$. In each
round, the $i$th learner commits to a decision $x^{i}_t \in \calK_i$. We define the joint decision $x_t = x^{1:k}_t\in \calK = \calK^{1:k}$. Afterwards, the learner incurs a loss $\ell_t(x_t)$ where $\ell_t$ is a convex loss function $\ell_t: \calK \rightarrow \RR$. The learner observers $\ell^i_t: \calK_i \rightarrow \RR$ defined by $\ell^i_t(x) =\ell_t(x, x^{-i}_{t})$, where we use the notation $(x, x^{-i}_{t})$ to represent $(x^1_{t}, \dots x^{i-1}_t, x, x^{i+1}_t, \dots x^{k}_t)$.

Regret is defined to be the total loss incurred by the algorithm with respect
to the loss of the best fixed single prediction found in hindsight for each player. Formally,
the average {\bf  multi-agent regret } of a set of learning algorithms $\calA_{1:k} = \calA_1, \dots \calA_k$ is defined as
\begin{align*}
    \regret_T(\calA_{1:k}) &\eqdef
    \sup_{\ell_1 \dots \ell_t} \bigg \{\frac{1}{T}\sum_{t=1}^T\ell_t(x_t) -
    \min_{x^{*}\in \calK}\frac{1}{T}\sum_{t=1}^T\ell_t(x^{*}) \bigg\} ~.
\end{align*}

Our goal is to create a black-box mechanism that takes as input regret minimizing strategies, and allows for joint-regret minimization with minimal or no central coordination.  

\subsection{Challenges of  multi-agent OCO} \label{rmk:linearlization}

A natural first attempt for joint regret minimization in  multi-agent OCO is to allow each regret minimizing agent to operate on its own on the given loss function.

In this subsection we show that this naive strategy fails.
As an example, consider the simple scalar two player game.  Let $\ell_t(x^1,x^2) = (x^1-x^2)^2 + 0.1 \|(x^1,x^2)\|_2^2$ for all $t$. Now we consider $\calA_1$ and $\calA_2$ to be algorithms that each alternate between playing $1$ on odd $t$ and $-1$ on even $t$.  
In this setup, the online learners see losses:
\begin{align*}
    \ell_t^1(z) = \ell_t^2(z)=\begin{cases}
          (z-1)^2 + 0.1 z^2 +0.1   & t  \ \text{is odd} \\
          (z+1)^2 + 0.1 z^2 + 0.1 &  t \ \text{is even}
     \end{cases}~
\end{align*}
By design of $\calA_1$ and $\calA_2$, predictions alternate between $(1,1)$ and $(-1, -1)$ and $\ell_t^1(x^1_t)= \ell_t^2(x^2_t)=0.2$. In contrast, the best fixed decision in hindsight is $0$ for both players, suffering loss $\ell_t^1(0)= \ell_t^2(0)=1.1$ for all $t$, so each player has \emph{negative} regret. However, for the multiplayer game, the best loss in hindsight is $0$ by playing $x^1_t = x^2_t=0$ for all $t$, so the multi-agent regret is $0.2 T$ . We note that this phenomenon is well known in the game theory literature, where failure of no-regret algorithms to converge to Nash equilibrium has been studied \cite{blum2008regret}. Our study extends this observation to the dynamic games setting.

\subsection{Reduction from OCO to  multi-agent OCO}

To overcome the lower bound of the previous section, we choose to apply the reduction from  multi-agent to single agent OCO via the instantaneous linearization. This simple change makes all the difference, and allows for provable sublinear  multi-agent regret, as we now prove. 

\begin{algorithm}
  \caption{Multiplayer OCO}
    \label{algorithm:multioco}
\begin{algorithmic}
    \STATE Input: OCO algorithm $\calA_i$, convex domain $\calK_i \subseteq \RR^{d}$\\
    \FOR {$t=1$  {\bfseries to} $T$}
    \STATE Learner $\calA_i$ predicts $x^i_t$
    \STATE Calculate $g^{i}_t =  \nabla \ell^{i}_t(x^{i}_{t})$
    \STATE Feed $\calA_i$ linear loss function $\ip{g^{i}_t}{\cdot}$
    \ENDFOR
\end{algorithmic}
\end{algorithm}

\begin{theorem}
    Suppose for each $i \in [k]$, OCO algorithm $\calA_i$ has regret $\regret_T(\calA_i)$ over $\calK_i \subseteq \RR^d$ with linear loss functions provided by Alg.~\ref{algorithm:multioco}, then the  multi-agent regret of Alg.~\ref{algorithm:multioco} is upper bounded by $\regret_T(\calA_{1:k}) \le \sum_{i=1}^k \regret_T(\calA_i)$.
\end{theorem}
\begin{proof}
Let $\bar{x} \in \argmin_{x\in \calK}\sum_{t=1}^T\ell_t(x)$ with $\bar{x}^i$ the component corresponding to player $i$. By the convexity of $\ell_t$,
\begin{align*}
   \regret_T(\calA_{1:k}) = \sum_{t=1}^T\left(\ell_t(x_t) -
   \ell_t(\bar{x})\right) \le \sum_{t=1}^T \nabla \ell_t(x_t)^{\top}(x_t - \bar{x})~.
\end{align*}
Now we note that we can decompose $\nabla \ell_t$ by player, so
\begin{align*}
    &
    \nabla\ell_t(x_t)^{\top} = 
           [\nabla\ell^1_t(x^1_t)^{\top}
           \dots
           \nabla\ell_t^k(x^k_t)^{\top}]
           \\
          &
          \Rightarrow \nabla \ell_t(x_t)^{\top}(x_t - \bar{x}) = \sum_{i=1}^k \nabla \ell^{i}_t(x^{i}_t)^{\top}(x^i_t - \bar{x}^{i})~.
\end{align*}
Now we note that, $\regret_T(\calA_i) = \sum_{t=1}^T\nabla \ell^{i}_t(x^{i}_t)^{\top}(x^i_t - \bar{x}^{i})$ and the result follows.
\end{proof}

\subsection{Multiplayer OCO with Memory}\label{sec:multioco_mem}

The control setting which motivates our whole study is inherently state-based. Therefore we need to extend the multi-agent OCO to allow for memory.

This is similar to the previous section but the loss functions now have \emph{memory} of the $h$ previous controls. After each of the $k$ online learners commits to the joint decision $x_t$ as defined in the previous section. The learners incur a loss $\ell_t(x_{t-h:t})$ where $\ell_t$ is a convex loss function $\ell_t: \calK^{h+1} \rightarrow \RR$. The learner observers $\ell^i_t: \calK_i^{h+1} \rightarrow \RR$ by $\ell^i_t(z) =\ell_t(z, x^{-i}_{t-h:t})$.

Regret is defined to be the total loss incurred by the algorithm with respect
to the loss of the best fixed single prediction found in hindsight for each player. Formally,
the average regret of a set of learning algorithms $\calA_{1:k}$ is
\begin{align*}
    \regret_T(\calA_{1:k}) &\eqdef
    \sup_{\ell_1 \dots \ell_t} \bigg \{\frac{1}{T}\big(\sum_{t=1}^T\ell_t(x_{t-h:t}) -
    \min_{x^{*}\in \calK}\sum_{t=1}^T\bar{\ell_t}(x^{*})\big) \bigg\} 
\end{align*}
where $\bar{\ell_t}(x) = \ell_t(x, \dots x)$.

\begin{algorithm}
 \caption{Multiplayer OCO with Memory}
    \label{algorithm:multioco_mem}
   
\begin{algorithmic}
    \STATE Input: OCO-M algorithm $\calA_i$, convex domain $\calK_i \subseteq \RR^{d}$ and memory length $h$\\
    \FOR {$t=1$  {\bfseries to} $T$}
    \STATE Learner $\calA_i$ predicts $x^i_t$
    \STATE Calculate $g^{i}_t =  \nabla \ell^{i}_t(x^{i}_{t-h:t})$
    \STATE Feed $\calA_i$ linear loss function $\ip{g^{i}_t}{\cdot}$
    \ENDFOR
\end{algorithmic}
\end{algorithm}

The main result of this section is given by the following theorem, which is formally proved in the appendix. The proof follows the memoryless setting in the previous subsection. 
\begin{theorem}\label{thm:multi_ocomem}
    Suppose for each $i \in [k]$ OCO-M algorithm with memory $\calA_i$ has regret $\regret_T(\calA_i)$ over $\calK_i \subseteq \RR^d$ with linear loss functions provided by Alg.~\ref{algorithm:multioco_mem}, then the regret of Alg.~\ref{algorithm:multioco_mem} has regret upper bounded by
    \begin{align*}
        \regret_T(\calA_{1:k}) \le \sum_{i=1}^k \regret_T(\calA_i)~.
    \end{align*}
\end{theorem}
\vspace{-4mm} See Appendix~\ref{sec:additional_proofs} for proof.

\section{The Meta-Algorithm and its Analysis}\label{sec:control_v2}

This section contains our main result: a meta-algorithm for multi-agent nonstochastic control. We start by defining the precise requirements from the single agent controllers that are given as input, and assumptions on them. We then proceed to describe the meta-algorithm and its main performance guarantee on the multi-agent regret. We proceed to describe a few special cases of interest. 

\subsection{Assumptions and definitions}

The single agent control algorithms that are the basis for the reduction need to satisfy following requirements:
\vspace{-3mm}
\begin{itemize}
    \item Each agent should be able to evaluate their policy given the controls of the other agents. 
    \item The joint cost function over all agents is jointly convex. 
    \item Agent policies need to be decoupled in the sense of Assumption~\ref{assumption:dist_policy}.
\end{itemize}
\vspace{-3mm}We now specify these requirements more formally. 

We introduce two related notions of a $(\eps, h)$-policy evaluation oracle (PEO) for a multi-agent control problem.  A joint PEO allows us to counterfactually evaluate a complete multi-agent policy, while a local PEO allows us to counterfactually evaluate the policy of a single agent, fixing the remainder of the controls. $\eps$ is the accuracy required of this oracle, while $h$ is something like a mixing time such that the cost of the $t$th state can be well approximated regardless of the history before $t-h$.  

Given the available control information $u_{1:t-1}$, observations
$y^i_{1:t}$, and cost function $c_t$, a loss function $\ell_t$ is a joint $(\eps, H)$-PEO, if $\ell_t(\theta_{1:H}^{1:k})$ provides an $\eps$ approximate estimate of
$c_t(\tilde{x}_t, \tilde{u}_t)$ where $\tilde{x}_t$ is the counterfactual state had
$\tilde{u}_{t-r+1}$ been played according to joint policy $\pi_{\theta^{1:k}_r}$.

Likewise,
$\ell^i_t$ is the local PEO of the $i$th agent, if $\ell^{i}_t(\theta_{1:h})$ provides an
$\eps$ approximate estimate of $c_t(\tilde{x}_t, \tilde{u}_t)$ where $\tilde{u}_t$ and
$\tilde{x}_t$ are the counterfactual state and observation had $u^{i}_{t-r+1}$ been played
according to policy $\pi_{\theta_r}$ while all other controls $u^{-i}_{1:t-1}$ remain the
same. These oracles are defined formally in Definition~\ref{def:policy_eval_oracle} in Appendix~\ref{sec:control_app}.

\begin{remark}\label{rmk:multi_pobs}
In order for a local agent to be able to provide an accurate estimate of the cost using only partial observations and controls, intuitively the cost may need to be restricted in in its dependence on the fully observed state.  For example, if all agents' partial observations contain certain features from the state (e.g. the $x$-$y$-$z$ position of the center of gravity of drone but no other velocity or angular features), the cost can depend on these and an agent PEO may exist.
\end{remark}

A key observation is that for policies satisfying Assumption~\ref{assumption:dist_policy}, if 
$u^{-i}_{t-h+1:t}$ are generated by policies $\theta^{-i}_{t-h+1:t}$ respectively, then a local PEO and a joint PEO can be related via $\ell^i_t = \ell_t(\cdot, \theta^{-i}_{t:t-h+1})$.  Because all controls are determined directly from the disturbances, changes in agent $i$'s controls have no impact on the controls of other agents.  

We now assume that beyond an initial burn-in time $T_b$, we can construct such an $(\eps,h)$-local PEO from information available to a controller.  We also must guarantee that the PEOs we use are convex in order to use OCO analysis.

\begin{assumption}[Information]\label{assumption:burn-in}
For $t \ge T_b > h$, for each agent $i$, there exists a functional $\cL^i$ such that $\ell^i_t = \cL^i(c_t, u_{1:t}, y^i_{1:t})$ is a $(\eps,h)$-local PEO for policy class $\Theta_i$. 
\end{assumption}

\begin{assumption}[Convexity]\label{assumption:convexity}
There exists a $(\eps, h)$ joint PEO $\ell_t$ that is convex such that $\ell^i_t = \cL^i(c_t, u_{1:t}, y^i_{1:t})$ satisfies $\ell^i_t = \ell_t(\cdot, \theta^{-i}_{t-h+1:t})$.
\end{assumption}
\subsection{Multi-Agent Control Meta-Algorithm and Analysis }

Our reduction is described in Algorithm \ref{algorithm:multiagent_ctrl_generic}.  It accepts as input OCO-M algorithms $\cA_i$ with requirements specified before, and applies them sequentially using the linearization technique of the previous section. 

\begin{algorithm}[!htp]
\caption{Multi-Agent Nonstochastic Control
\label{algorithm:multiagent_ctrl_generic}}
\begin{algorithmic}[1]
\STATE \textbf{Input:} Horizon length $T$, burn in time $T_b$, policy evaluation horizon $h$, and for each $i \in [k]$, OCO-M algorithms $\calA_i$ over policy parameterization class $\Theta^i$, policy evaluation functional $\cL^i$ .
\STATE Initialize $\theta^i_0 =\dots = \theta^i_{T_b}$ as a policy that always plays control $0$. 
\FOR {$t=1$ {\bfseries to} $T$}
\FOR {$i=1$ {\bfseries to} $k$}
\STATE Observe $y^i_{t}$
\STATE Play control $u^i_t = \pi_{\theta^i_t}(u_{1:{t-1}}, y^{i}_{1:t})$
\STATE Observe controls from other agents to form joint control $u_{t}= u^{1:k}_t$ 
\STATE Observe cost $c_t$ 
\IF{$t > T_b$}
\STATE Construct policy evaluation oracle $\ell^i_t = \cL^i(c_t, u_{1:{t}}, y^i_{1:t})$
\STATE Compute $g^i_t = \nabla \ell_t^i(\theta^i_{t-h, t})$
\STATE Feed $\calA_i$ linear loss function  $\ip{g^{i}_t}{\cdot}$ and receive policy parameterization $\theta^i_{t+1}$.
\ENDIF
\ENDFOR 
Pay $c_t(x_{t}, u_{t})$
\ENDFOR
\end{algorithmic}
\end{algorithm}

The inner loop of the algorithm starting at Line 5 is what is performed by a local agent.  The information available in this scope is the local observation and the controls of the other agents.  We note that Alg.~\ref{algorithm:multiagent_ctrl_generic} only uses $\ell^i_t$ instead of $\ell_t$, which is only used for analysis.  This is fundamental as each agent is unaware of the policies of other agents.  

The performance guarantee for this meta-algorithm is given by:

\begin{theorem}\label{thm:ctrl_regret}
Suppose Assumptions~\ref{assumption:bounded_nat_y},~\ref{assumption:cost},~\ref{assumption:dist_policy} hold and Assumptions~\ref{assumption:burn-in}, ~\ref{assumption:convexity} holds with burn-in $T_b$, horizon $h$ and error $\eps$. If $\cA_i$ are OCO-M algorithms with regret $\regret_T(\cA_i)$ on linear loss functions provided by Algorithm~\ref{algorithm:multiagent_ctrl_generic} over policy parameterization class $\Theta^i$, then the  multi-agent regret of   Algorithm~\ref{algorithm:multiagent_ctrl_generic}  is bounded by 
\begin{align*}
    \regret_T(\cC) \le  \sum_{i =1}^k \regret_T(\cA_i) + \frac{T_b C R_{\text{nat}}^2}{T} + 2\eps ~.
\end{align*}
\end{theorem}

\begin{proof}
In order to analyze the regret we begin with the following regret decomposition.  
We first let $\theta^{*} = \argmin_{\theta \in \Theta} \sum_{t=1}^T c_t(y^{\pi_{\theta}}_t, u^{\pi_{\theta}}_t)$
and let $\pi^{*} = \pi_{\theta^{*}}$ denote offline optimal joint policy from our comparator policy class.
\begin{align}
    &\regret_T(\cC) = \frac{1}{T} \sum_{t=1}^T c_t(x^{\cC}_t, u^{\cC}_t) - \frac{1}{T}\sum_{t=1}^T c_t(x^{\pi^{*}}_t, u^{\pi^{*}}_t) \label{eq:regret_decomp_gen}\\
    &\le \underbrace{\bigg(\frac{1}{T}\sum_{t=1}^{T_b} c_t(x^{\text{nat}}_t, 0) - \sum_{t=1}^{T_b} c_t(x^{\pi_{\theta^{*}}}_t, u^{\pi_{\theta^{*}}}_t)\bigg)}_{\text{burn-in cost}}\nonumber\\
    &+\underbrace{\bigg(\frac{1}{T}\sum_{t=T_b + 1}^T c_t(x^{\cC}_t, u^{\cC}_t) - \frac{1}{T}\sum_{t=T_b+1}^T \ell_t(\theta_{t:t-h})\bigg)}_{\text{algorithm truncation error}} \nonumber\\
    &+ \underbrace{\bigg(\frac{1}{T}\sum_{t=T_b+1}^T \ell_t(\theta_{t:t-h}) - \frac{1}{T}\sum_{t=T_b+1}^T \ell_t(\theta^{*}, \dots, \theta^{*})\bigg)}_{\text{multi-agent $\ell$ policy regret}} \nonumber \\
    &+ \underbrace{\bigg(\frac{1}{T}\sum_{t=T_b+1}^T \ell_t(\theta^{*}, \dots, \theta^{*}) - \frac{1}{T}\sum_{t=T_b+1}^T c_t(x^{\pi_{\theta^{*}}}_t, u^{\pi_{\theta^{*}}}_t \bigg)}_{\text{comparator truncation error}}\nonumber
\end{align}

The first term of the decomposition is exactly the cost of the first $T_b$ iterations because the Alg.~\ref{algorithm:multiagent_ctrl_generic} starts with $T_b$ zero controls so by definition $x^{\cC}_t$ are Nature's $x$'s.  By Assumption~\ref{assumption:cost}, losses are nonnegative so we have $-\sum_{t=1}^{T_b} c_t(x^{\pi_{\theta^{*}}}_t, u^{\pi_{\theta}}_t) \leq 0$ so we drop these terms.
By combining Assumptions~\ref{assumption:cost} and \ref{assumption:bounded_nat_y} we bound the burn-in cost as
\begin{align*}
    \frac{1}{T}\sum_{t=1}^{T_b} c_t(x^{\text{nat}}_t, 0) \le \frac{C}{T}\sum_{t=1}^{T_b}\|x_t\|^2 \le \frac{C T_b R^2_{\text{nat}}}{T}~.
\end{align*}

The comparator truncation error is bounded by $\eps$ because for $t > T_b$, $|\ell_t(\theta^{*}, \dots, \theta^{*}) - c_t(x^{\pi_{\theta^{*}}}_t, u^{\pi_{\theta^{*}}}_t)| \le \eps$  by definition of an $(\eps,h)$ joint PEO.  Similarly, since $\cC$ produces $x^{\cC}_t$ with controls $u_s$ from joint policy $\theta_s$ for the preceding $h$ controls, the algorithm truncation error is also bounded by $\eps$.

To bound the $\ell_t$ multi-agent policy regret term, we note that
Algorithm~\ref{algorithm:multiagent_ctrl_generic} chooses $\theta_t$ exactly by the mechanism
in Algorithm~\ref{algorithm:multioco_mem} with loss functions $\ell^i_t$. By
Assumption~\ref{assumption:burn-in}, $\ell^i_t= \ell_t(\cdot, \theta^{-i}_{t-h:t-1})$ and
by Assumption~\ref{assumption:convexity}, $\ell_t$ is convex as desired. As such, we bound the $\ell_t$ multi-agent policy regret term
using Theorem~\ref{thm:multi_ocomem} as
\begin{align*}
\sum_{t=T_b}^T \ell_t(\theta_{t:t-h}) - \sum_{t=T_b+1}^T \ell_t(\theta^{*}, \dots, \theta^{*}) \le T\sum_{i =1}^k \regret_T(\cA_i) ~.
\end{align*}

The result follows after combining the above bounds.
\end{proof}

\vspace{-2mm}\subsection{Applications to Linear Dynamical Systems}
Algorithm~\ref{algorithm:multiagent_ctrl_generic} leaves most of the technical challenge to a policy evaluation oracle. It's not immediately obvious that such a function exists and is efficient to compute. Fortunately, recent work from nonstochastic control provides exactly  what we need for Linear Dynamical Systems with a variety of policy classes.

For simplicity of presentation, we consider only stable systems. However, we note that in settings with a single shared observation $y_t$, these algorithms can be extended to work for unstable systems.  In order to do this, all agents must be aware of baseline stabilizing controllers employed by each agent, and Algorithm~\ref{algorithm:multiagent_ctrl_generic} can then be applied on the stable closed-loop system including the stabilizing component. For a thorough treatment of this approach, we refer the reader to Appendix G of \cite{simchowitz2020improper}.

We assume strong stability, which guarantees spectral radius of $A$ strictly less than $1$. Strong stability assures us that impact of disturbances and controls from the far history have a negligible impact on the current state.  This geometric decay enables the construction of PEO's with horizon $h$ just logarithmic in $T$ while providing $\eps = \frac{1}{T}$ approximation error.

\paragraph{Fully Observed LDS with Disturbance Action Control}
We consider the case of a fully observed LDS  defined as follows:
\vspace{-4mm}\begin{align*}
    y_{t+1} = x_{t+1} = f_t(x_t,u_t) + w_t = A x_t + B u_t + w_t~, 
\end{align*}
with bounded adversarial disturbances $w_t$, known strongly stable dynamics, and well behaved convex costs $c_t$ satisfying Assumption~\ref{assumption:cost}\footnote{In general, we may actually need more assumptions on the costs, and the size of our comparator policy classes as $\|g^i_t\|$ in Line 11 of Algorithm~\ref{algorithm:multiagent_ctrl_generic} depends on these pieces, and hence so does $\regret_T(\cA_i)$.}.

We now show that a linear system with a DAC policy class $\cM$ can satisfy assumptions required for Theorem~\ref{thm:ctrl_regret}.  We already saw that Assumption~\ref{assumption:dist_policy} holds for DAC policies.  Next we show that Assumption~\ref{assumption:bounded_nat_y} holds. By fully unrolling dynamics, Nature's $x$'s can be written as a linear combination of disturbances $x^{\text{nat}}_t = \sum_{s=0}^{t-1} A^s w_{t-s}$.

Now because $w_t$ are bounded, and $A$ is strongly stable, $\|x^{\text{nat}}_t\|$ can be bounded by a convergent geometric series. The last piece is constructing a PEO for a DAC policy.  Unrolling dynamics, we express states in terms of controls:
\vspace{-1mm}
\begin{align*}
x_t &\approx x^{\text{nat}}_t + [B, AB, \dots A^{h-1} B]u_{t-1:t-h} = {\hat{G}}u_{t-1:t-h}~.
\end{align*}
The truncation follows from strong stability of the dynamics.  Decomposing $\hat{G}$, our Markov operator by agent, we can write
$
    x_t \approx x^{\text{nat}}_t + \sum_{i=1}^k \hat{G}_i u^{i}_{t-1:t-h}.
$ Now to compute $\ell^i_t(M_{1:h}) = c_t(\tilde{x}(M_{1:h}), M w_{t-m:t-1})$ we hold $u^{-i}_s$ fixed while changing the policies for agent $i$.  Because the counterfactual state $\tilde{x}(M_{1:h})$ is affine in $ u^{i}_{t-1:t-h}$ and by the DAC policy parameterization agent $i$'s controls are linear, the counterfactual state can be approximated as linear in $M_{1:h}$. Composing with a convex cost $c_t$ gives a convex PEO $\ell^i_t$.  A complete analysis in the single agent setting can be found in \cite{agarwal2019online}.  In order to compute $\ell^i_t(M_{1:h})$, the disturbances $w_t$ are needed.  Because agent $i$ has access to the joint control $u_t$, disturbances can be computed as $w_t = x_{t+1} - Ax_t - Bu_t$.  Once $T_b = m+h$ disturbances are observed, $\ell^i_t$ can be computed.

Similar ideas can be applied using Nature's $y$'s and DRC to handle partial observations. Furthermore, Theorem~\ref{thm:ctrl_regret} can be applied with richer neural network policy classes as well.  For a more complete discussion, see Appendix~\ref{sec:LDS_applications}.

\vspace{-3mm}\section{Experiments} \label{sec:exp}

\vspace{-2mm}\subsection{Robustness to System Failures}\label{sec:sys_fail}
Our algorithm guarantee can also be extended to handle situations where certain agents no longer are locally regret minimizing (e.g. an actuator breaks).  Because Algorithm~\ref{algorithm:multiagent_ctrl_generic} is agnostic to the policy played by other players, if we want to bound the regret for a subset $I \subset [k]$ of controllers, we can just view the policy class of the remaining agents as singleton open-loop policies that always play the same controls and so Theorem~\ref{thm:ctrl_regret} will be a regret guarantee with respect to just the agents $I$ with the remaining controls obliviously fixed.

Such robust guarantees are not achieved with existing approaches in nonstochastic control.  Typically, as in the Gradient Perturbation Controller (GPC) \cite{agarwal2019online}, a PEO $\ell_t(\theta_{1:h})$ is passed to an online learning algorithm.  If the system breaks down and certain controls are not played according to the chosen policy, there will be a mismatch between the true cost and the PEO.  In contrast, because $\ell^i_t$ uses the actual controls played by the algorithm, for an agent $i$ that is not experiencing failure, the oracle accuracy remains unchanged and provides the correct gradient signal.

As such, Algorithm~\ref{algorithm:multiagent_ctrl_generic} may be useful to provide robustness even for a single agent system with some failures. 

\vspace{-2mm}\subsection{ADMIRE experiments}
In Fig.~\ref{fig:admire} we run GPC, LQR control, and a linear $\cH_{\inf}$ control policy, along with Multi-Agent GPC (MAGPC), which is Algorithm~\ref{algorithm:multiagent_ctrl_generic} using gradient descent as the online learners $\cA^i$. These are used to stabilize the ADMIRE overactuated aircraft model \cite{harkegaard2005resolving} in the presence of $3$ different disturbance profiles: Gaussian noise, a random walk, and a sinusoidal disturbance pattern. We run the algorithms on the complete system, and with the $4$th controller forced inactive by zeroing out the control to demonstrate the robustness of our approach.  Figure~\ref{fig:admire} demonstrates the GPC failure described in Section~\ref{sec:sys_fail}.  In the random walk experiment, GPC explodes, but MAGPC stabilizes the system. See Appendix~\ref{sec:admire} for full details.
We note that, when a control failure does not occur and learning rates are sufficiently small, GPC and MAGPC are essentially the same algorithm. This is because the policy parameterizations moves slowly so $\ell^i_t(\theta_{t-h:t}) \approx \ell^i_t(\theta_t \dots \theta_t)$, which is GPC's proxy loss.

\begin{figure}
    \centering
    \begin{subfigure}[b]{\linewidth}
    \centering
    \includegraphics[width=\linewidth]{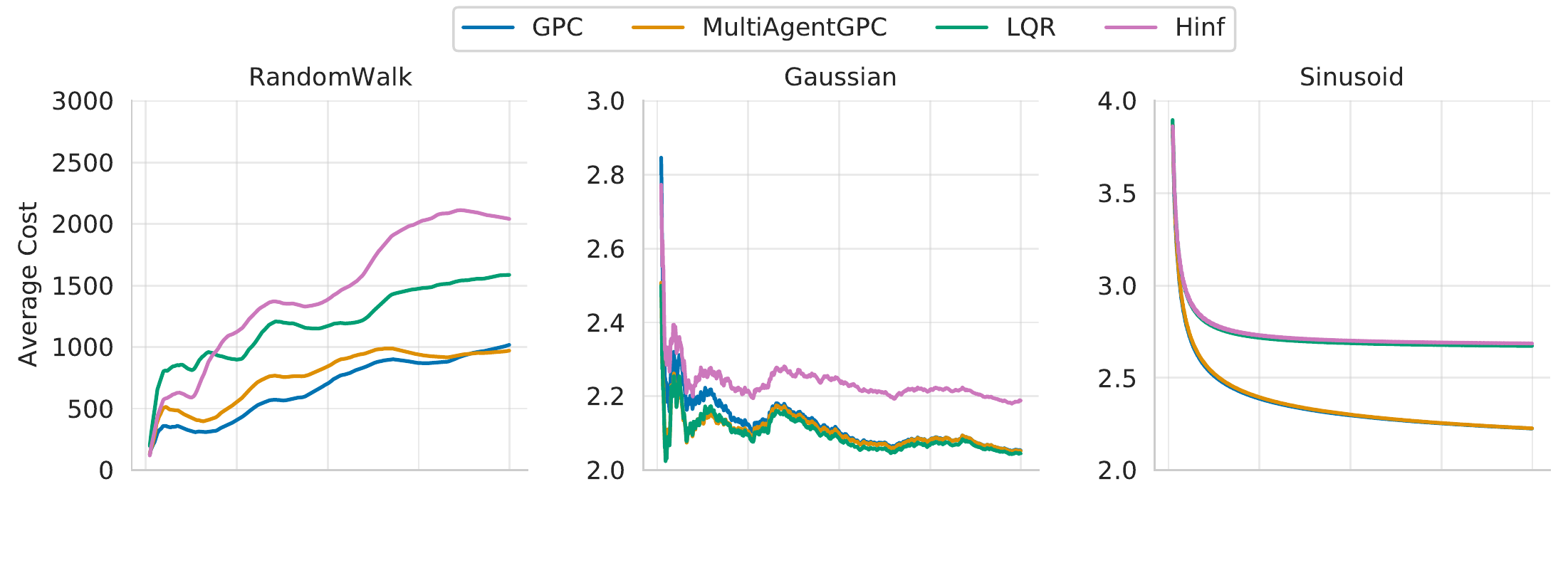}
    \vspace{-25pt}
    \end{subfigure}
    \begin{subfigure}[b]{\linewidth}
    \centering
    \includegraphics[width=\linewidth]{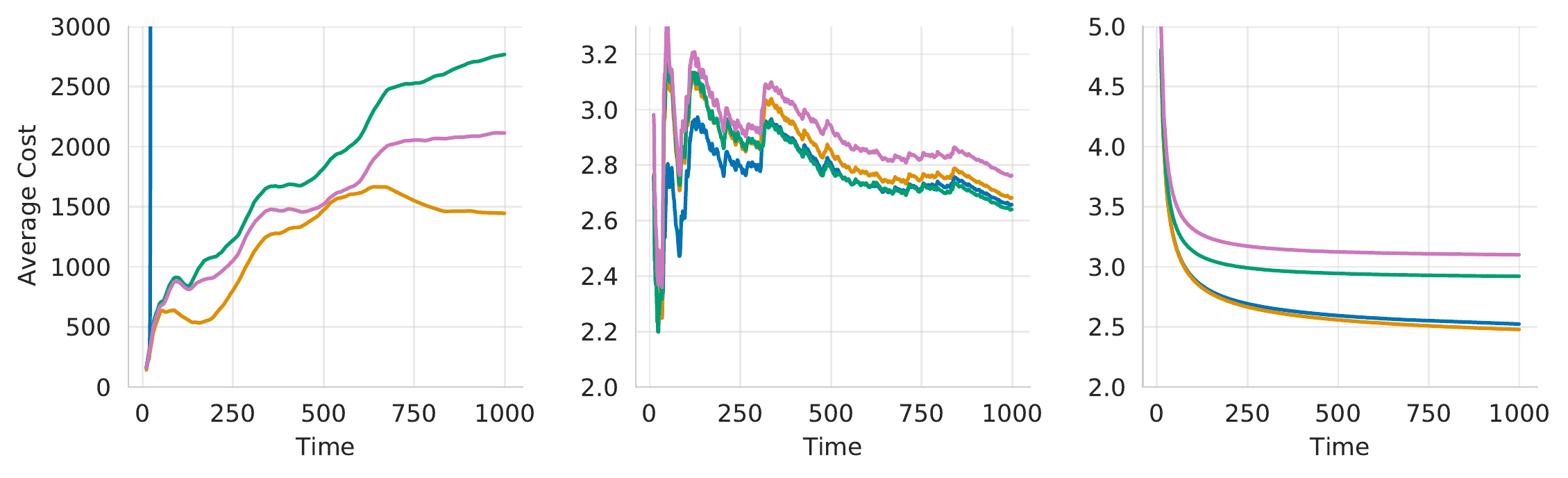}
    \vspace{-20pt}
    \end{subfigure}
    \caption{Cost of GPC, MultiAgentGPC, $\cH_2$ and $\cH_{\infty}$ control before and after the $4$th controller becomes inactive in the presence of gaussian disturbances, disturbances generated by a random walk, and sinusoidal disturbances.}
    \label{fig:admire}
     \vspace{-15pt}
\end{figure}
\section{Discussion and Conclusions}

\paragraph{Conclusions}
We have described a new approach for regret minimization in multi-agent control that is based on a new proposed performance metric: multi-agent regret. This measures the difference in cost between that of a distributed control system and that of the best joint policy from a reference class. We give a reduction that takes any regret minimizing controller and converts it into a regret minimizing multi-agent distributed controller.

\vspace{-1mm}\paragraph{Future directions}
The most interesting question is whether similar metrics of performance and ideas can be applied to general multi-agent reinforcement learning (MARL), since here we exploit full knowledge of the dynamics. It will be interesting to see if our information model can be relaxed and still allow sublinear multi-agent regret.
It also remains to  explore the extent of robustness obtained from a multi-agent regret minimizing algorithm.







\bibliography{icml_main}

\begin{thebibliography}{40}
\providecommand{\natexlab}[1]{#1}
\providecommand{\url}[1]{\texttt{#1}}
\expandafter\ifx\csname urlstyle\endcsname\relax
  \providecommand{\doi}[1]{doi: #1}\else
  \providecommand{\doi}{doi: \begingroup \urlstyle{rm}\Url}\fi

\bibitem[Abbasi-Yadkori \& Szepesv{\'a}ri(2011)Abbasi-Yadkori and
  Szepesv{\'a}ri]{abbasi2011regret}
Abbasi-Yadkori, Y. and Szepesv{\'a}ri, C.
\newblock Regret bounds for the adaptive control of linear quadratic systems.
\newblock In \emph{Proceedings of the 24th Annual Conference on Learning
  Theory}, pp.\  1--26, 2011.

\bibitem[Agarwal et~al.(2019)Agarwal, Bullins, Hazan, Kakade, and
  Singh]{agarwal2019online}
Agarwal, N., Bullins, B., Hazan, E., Kakade, S., and Singh, K.
\newblock Online control with adversarial disturbances.
\newblock In \emph{International Conference on Machine Learning}, pp.\
  111--119. PMLR, 2019.

\bibitem[Blum et~al.(2008)Blum, Hajiaghayi, Ligett, and Roth]{blum2008regret}
Blum, A., Hajiaghayi, M., Ligett, K., and Roth, A.
\newblock Regret minimization and the price of total anarchy.
\newblock pp.\  373--382, 05 2008.
\newblock \doi{10.1145/1374376.1374430}.

\bibitem[Bullo et~al.(2009)Bullo, Cortes, and Martinez]{Bullo2009}
Bullo, F., Cortes, J., and Martinez, S.
\newblock \emph{Distributed control of robotic networks}.
\newblock Princeton University Press, 2009.

\bibitem[Cao et~al.(2013)Cao, Yu, Ren, and Chen]{Cao2013}
Cao, Y., Yu, W., Ren, W., and Chen, G.
\newblock An overview of recent progress in the study of distributed
  multi-agent coordination.
\newblock \emph{IEEE Transactions on Industrial Informatics}, 9\penalty0
  (1):\penalty0 427--438, 2013.
\newblock \doi{10.1109/TII.2012.2219061}.

\bibitem[Carli et~al.(2008)Carli, Chiuso, Schenato, and
  Zampieri]{ruggero2008DistributedKalman}
Carli, R., Chiuso, A., Schenato, L., and Zampieri, S.
\newblock Distributed kalman filtering based on consensus strategies.
\newblock \emph{IEEE Journal on Selected Areas in Communications}, 26\penalty0
  (4):\penalty0 622--633, 2008.
\newblock \doi{10.1109/JSAC.2008.080505}.

\bibitem[Cesa-Bianchi \& Lugosi(2006)Cesa-Bianchi and
  Lugosi]{cesa2006prediction}
Cesa-Bianchi, N. and Lugosi, G.
\newblock \emph{Prediction, learning, and games}.
\newblock Cambridge university press, 2006.

\bibitem[Chen et~al.(2021)Chen, Minasyan, Lee, and Hazan]{chen2021provable}
Chen, X., Minasyan, E., Lee, J.~D., and Hazan, E.
\newblock Provable regret bounds for deep online learning and control, 2021.

\bibitem[Christofides et~al.(2013)Christofides, Scattolini, de~la Pena, and
  Liu]{christofides2013distributed}
Christofides, P.~D., Scattolini, R., de~la Pena, D.~M., and Liu, J.
\newblock Distributed model predictive control: A tutorial review and future
  research directions.
\newblock \emph{Computers \& Chemical Engineering}, 51:\penalty0 21--41, 2013.

\bibitem[Das \& Moura(2017)Das and Moura]{das2017consensusinnovations}
Das, S. and Moura, J. M.~F.
\newblock Consensus+innovations distributed kalman filter with optimized gains.
\newblock \emph{IEEE Transactions on Signal Processing}, 65\penalty0
  (2):\penalty0 467--481, 2017.
\newblock \doi{10.1109/TSP.2016.2617827}.

\bibitem[Foerster et~al.(2018)Foerster, Chen, Al-Shedivat, Whiteson, Abbeel,
  and Mordatch]{foerster2018learning}
Foerster, J., Chen, R.~Y., Al-Shedivat, M., Whiteson, S., Abbeel, P., and
  Mordatch, I.
\newblock Learning with opponent-learning awareness.
\newblock In \emph{Proceedings of the 17th International Conference on
  Autonomous Agents and MultiAgent Systems}, pp.\  122--130, 2018.

\bibitem[Ghai et~al.(2021)Ghai, Snyder, Majumdar, and
  Hazan]{ghaiAdvDisturbance}
Ghai, U., Snyder, D., Majumdar, A., and Hazan, E.
\newblock Generating adversarial disturbances for controller verification.
\newblock In \emph{Proceedings of the 3rd Conference on Learning for Dynamics
  and Control}, volume 144 of \emph{Proceedings of Machine Learning Research},
  pp.\  1192--1204. PMLR, 07 -- 08 June 2021.

\bibitem[Goel \& Hassibi(2021)Goel and Hassibi]{goel2021competitive}
Goel, G. and Hassibi, B.
\newblock Competitive control.
\newblock \emph{arXiv preprint arXiv:2107.13657}, 2021.

\bibitem[Gradu et~al.(2020)Gradu, Hazan, and Minasyan]{gradu2020adaptive}
Gradu, P., Hazan, E., and Minasyan, E.
\newblock Adaptive regret for control of time-varying dynamics.
\newblock \emph{arXiv preprint arXiv:2007.04393}, 2020.

\bibitem[H{\"a}rkeg{\aa}rd \& Glad(2005)H{\"a}rkeg{\aa}rd and
  Glad]{harkegaard2005resolving}
H{\"a}rkeg{\aa}rd, O. and Glad, S.~T.
\newblock Resolving actuator redundancy—optimal control vs. control
  allocation.
\newblock \emph{Automatica}, 41\penalty0 (1):\penalty0 137--144, 2005.

\bibitem[Hart \& Mas-Colell(2000)Hart and Mas-Colell]{hart2000simple}
Hart, S. and Mas-Colell, A.
\newblock A simple adaptive procedure leading to correlated equilibrium.
\newblock \emph{Econometrica}, 68\penalty0 (5):\penalty0 1127--1150, 2000.

\bibitem[Hazan(2019)]{hazan2019introduction}
Hazan, E.
\newblock Introduction to online convex optimization.
\newblock \emph{arXiv preprint arXiv:1909.05207}, 2019.

\bibitem[Hazan \& Singh(2021)Hazan and Singh]{IcmlTutorial21}
Hazan, E. and Singh, K.
\newblock Tutorial: online and non-stochastic control, July 2021.

\bibitem[Jin et~al.(2021)Jin, Liu, Wang, and Yu]{jin2021v}
Jin, C., Liu, Q., Wang, Y., and Yu, T.
\newblock V-learning--a simple, efficient, decentralized algorithm for
  multiagent rl.
\newblock \emph{arXiv preprint arXiv:2110.14555}, 2021.

\bibitem[Knorn et~al.(2016)Knorn, Chen, and Middleton]{Knorn2016}
Knorn, S., Chen, Z., and Middleton, R.~H.
\newblock Overview: Collective control of multiagent systems.
\newblock \emph{IEEE Transactions on Control of Network Systems}, 3\penalty0
  (4):\penalty0 334--347, 2016.
\newblock \doi{10.1109/TCNS.2015.2468991}.

\bibitem[Lefebvre et~al.(1980)Lefebvre, Richter, and DeCarlo]{lefebvre1980}
Lefebvre, S., Richter, S., and DeCarlo, R.
\newblock Decentralized control of linear interconnected multivariable systems.
\newblock In \emph{1980 19th IEEE Conference on Decision and Control including
  the Symposium on Adaptive Processes}, pp.\  525--529, 1980.
\newblock \doi{10.1109/CDC.1980.271852}.

\bibitem[Matei \& Baras(2010)Matei and Baras]{matei2010consensusfilter}
Matei, I. and Baras, J.~S.
\newblock Consensus-based distributed linear filtering.
\newblock In \emph{49th IEEE Conference on Decision and Control (CDC)}, pp.\
  7009--7014, 2010.
\newblock \doi{10.1109/CDC.2010.5718072}.

\bibitem[Minasyan et~al.(2021)Minasyan, Gradu, Simchowitz, and
  Hazan]{minasyan2021online}
Minasyan, E., Gradu, P., Simchowitz, M., and Hazan, E.
\newblock Online control of unknown time-varying dynamical systems.
\newblock \emph{Advances in Neural Information Processing Systems}, 34, 2021.

\bibitem[Mnih et~al.(2016)Mnih, Badia, Mirza, Graves, Lillicrap, Harley,
  Silver, and Kavukcuoglu]{mnih2016asynchronous}
Mnih, V., Badia, A.~P., Mirza, M., Graves, A., Lillicrap, T., Harley, T.,
  Silver, D., and Kavukcuoglu, K.
\newblock Asynchronous methods for deep reinforcement learning.
\newblock In \emph{International conference on machine learning}, pp.\
  1928--1937. PMLR, 2016.

\bibitem[Nedić et~al.(2018)Nedić, Olshevsky, and Rabbat]{Nedic2018}
Nedić, A., Olshevsky, A., and Rabbat, M.~G.
\newblock Network topology and communication-computation tradeoffs in
  decentralized optimization.
\newblock \emph{Proceedings of the IEEE}, 106\penalty0 (5):\penalty0 953--976,
  2018.
\newblock \doi{10.1109/JPROC.2018.2817461}.

\bibitem[Olfati-Saber(2005)]{Olfati2005}
Olfati-Saber, R.
\newblock Distributed kalman filter with embedded consensus filters.
\newblock In \emph{Proceedings of the 44th IEEE Conference on Decision and
  Control}, pp.\  8179--8184, 2005.
\newblock \doi{10.1109/CDC.2005.1583486}.

\bibitem[Olfati-Saber et~al.(2007)Olfati-Saber, Fax, and Murray]{Olfati2007}
Olfati-Saber, R., Fax, J.~A., and Murray, R.~M.
\newblock Consensus and cooperation in networked multi-agent systems.
\newblock \emph{Proceedings of the IEEE}, 95\penalty0 (1):\penalty0 215--233,
  2007.
\newblock \doi{10.1109/JPROC.2006.887293}.

\bibitem[Oppenheimer et~al.(2006)Oppenheimer, Doman, and
  Bolender]{oppenheimer2006}
Oppenheimer, M.~W., Doman, D.~B., and Bolender, M.~A.
\newblock Control allocation for over-actuated systems.
\newblock In \emph{2006 14th Mediterranean Conference on Control and
  Automation}, pp.\  1--6, 2006.
\newblock \doi{10.1109/MED.2006.328750}.

\bibitem[Park \& Martins(2017)Park and Martins]{park2017distlti}
Park, S. and Martins, N.~C.
\newblock Design of distributed lti observers for state omniscience.
\newblock \emph{IEEE Transactions on Automatic Control}, 62\penalty0
  (2):\penalty0 561--576, 2017.
\newblock \doi{10.1109/TAC.2016.2560766}.

\bibitem[Salimans et~al.(2017)Salimans, Ho, Chen, Sidor, and
  Sutskever]{salimans2017evolution}
Salimans, T., Ho, J., Chen, X., Sidor, S., and Sutskever, I.
\newblock Evolution strategies as a scalable alternative to reinforcement
  learning.
\newblock \emph{arXiv preprint arXiv:1703.03864}, 2017.

\bibitem[Simchowitz et~al.(2020)Simchowitz, Singh, and
  Hazan]{simchowitz2020improper}
Simchowitz, M., Singh, K., and Hazan, E.
\newblock Improper learning for non-stochastic control.
\newblock In \emph{Conference on Learning Theory}, pp.\  3320--3436. PMLR,
  2020.

\bibitem[Sobel \& Shapiro(1985)Sobel and
  Shapiro]{sobelEigenstructureAssignment}
Sobel, K. and Shapiro, E.
\newblock Eigenstructure assignment for design of multimode flight control
  systems.
\newblock \emph{IEEE Control Systems Magazine}, 5\penalty0 (2):\penalty0 9--15,
  1985.
\newblock \doi{10.1109/MCS.1985.1104941}.

\bibitem[St{\"u}rz et~al.(2020)St{\"u}rz, Zhu, Rosolia, Johansson, and
  Borrelli]{sturz2020distributed}
St{\"u}rz, Y.~R., Zhu, E.~L., Rosolia, U., Johansson, K.~H., and Borrelli, F.
\newblock Distributed learning model predictive control for linear systems.
\newblock In \emph{2020 59th IEEE Conference on Decision and Control (CDC)},
  pp.\  4366--4373. IEEE, 2020.

\bibitem[Sutton \& Barto(2018)Sutton and Barto]{sutton2018reinforcement}
Sutton, R.~S. and Barto, A.~G.
\newblock \emph{Reinforcement learning: An introduction}.
\newblock MIT press, 2018.

\bibitem[Sutton et~al.(2000)Sutton, McAllester, Singh, and
  Mansour]{sutton2000policy}
Sutton, R.~S., McAllester, D.~A., Singh, S.~P., and Mansour, Y.
\newblock Policy gradient methods for reinforcement learning with function
  approximation.
\newblock In \emph{Advances in neural information processing systems}, pp.\
  1057--1063, 2000.

\bibitem[Tohidi et~al.(2017)Tohidi, Yildiz, and
  Kolmanovsky]{Tohidi2017AdaptiveCA}
Tohidi, S.~S., Yildiz, Y., and Kolmanovsky, I.~V.
\newblock Adaptive control allocation for over-actuated systems with actuator
  saturation.
\newblock \emph{IFAC-PapersOnLine}, 50:\penalty0 5492--5497, 2017.

\bibitem[Wang et~al.(2020)Wang, Fullmer, Liu, and Morse]{wang2020distcontrol}
Wang, L., Fullmer, D., Liu, F., and Morse, A.~S.
\newblock Distributed control of linear multi-channel systems: Summary of
  results.
\newblock In \emph{2020 American Control Conference (ACC)}, pp.\  4576--4581,
  2020.
\newblock \doi{10.23919/ACC45564.2020.9148043}.

\bibitem[Wang \& Davison(1973)Wang and Davison]{wang1973}
Wang, S.-H. and Davison, E.
\newblock On the stabilization of decentralized control systems.
\newblock \emph{IEEE Transactions on Automatic Control}, 18\penalty0
  (5):\penalty0 473--478, 1973.
\newblock \doi{10.1109/TAC.1973.1100362}.

\bibitem[Watkins \& Dayan(1992)Watkins and Dayan]{watkins1992q}
Watkins, C.~J. and Dayan, P.
\newblock Q-learning.
\newblock \emph{Machine learning}, 8\penalty0 (3-4):\penalty0 279--292, 1992.

\bibitem[Williams(1992)]{williams1992simple}
Williams, R.~J.
\newblock Simple statistical gradient-following algorithms for connectionist
  reinforcement learning.
\newblock \emph{Machine learning}, 8\penalty0 (3):\penalty0 229--256, 1992.

\end{thebibliography}
\bibliographystyle{icml2022}

\newpage
\appendix
\onecolumn
\section{Deferred Technical Material}\label{sec:additional_proofs}
\subsection{Proof of Theorem~\ref{thm:multi_ocomem}}
We provide a proof of Theorem~\ref{thm:multi_ocomem}, for the multi-agent regret of an OCO-M algorithm.
\begin{proof}
Let $\bar{x} \in \argmin_{x\in \calK}\sum_{t=1}^T\ell_t(x)$ with $\bar{x}^i$ the component corresponding to player $i$.  Now by the convexity of $\ell_t$,
\begin{align*}
   \regret_T(\calA_1, \dots \calA_k) = \sum_{t=1}^T\ell_t(x_{t-h:t}) -
   \sum_{t=1}^T\ell_t(\bar{x}, \dots \bar{x}) \le \sum_{t=1}^T \nabla \ell_t(x_{t-h:t})^{\top}(x_{t-h:t} - (\bar{x},\dots,\bar{x}))~.
\end{align*}

Now we note that we can decompose $\nabla \ell_t$ by player (after reorganizing coordinates), so
\begin{align*}
           \nabla \ell_t(x_{t-h:t}) =[\nabla\ell^1_t(x^1_{t-h:t})^{\top},
           \nabla\ell^2_t(x^2_{t-h:t})^{\top} ,\dots
           \nabla\ell_t^k(x^k_{t-h:t})^{\top}]\\
         \Rightarrow \nabla \ell_t(x_{t-h:t})^{\top}(x_{t-h:t} - (\bar{x},\dots,\bar{x})) = \sum_{i=1}^k \nabla \ell_t(x^{i}_{t-h:t})^{\top}(x^i_{t-h:t} - (\bar{x}^{i}, \dots ,\bar{x}^{i}))~.
\end{align*}
Now we note that, $\regret_T(\calA_i) = \sum_{t=1}^T\nabla \ell_t(x^{i}_{t-h:t})^{\top}(x^i_{t-h:t} - (\bar{x}^{i}, \dots ,\bar{x}^{i}))$ and the result follows.
\end{proof}
\subsection{Linear Dynamical Controller}\label{sec:LDC}
\begin{definition}[Linear Dynamic Controller]
\label{LDC:def}
A \emph{linear dynamic controller} $\pi$ is a linear dynamical system 
$(A_{\pi}, B_{\pi}, C_{\pi}, D_{\pi})$ with internal state $s_t \in \RR^{d_{\pi}}$, input $x_t \in \RR^{d_{x}}$, and output $u_t \in \RR^{d_u}$ that satisfies
\begin{align*}
 s_{t+1} = A_{\pi} s_t + B_{\pi} x_t,~ u_{t} = C_{\pi} s_t + D_{\pi} x_t \ .
\end{align*}
\end{definition}
Linear dynamical controllers involve an internal linear dynamical system to produce a control.  The combination of a Kalman Filter with Optimal control on the state estimates is a classic example.
\subsection{Technical Details from Section~\ref{sec:control_v2}}\label{sec:control_app}
Below are the formal definition of joint and local policy evaluation oracles:
\begin{definition}\label{def:policy_eval_oracle}
$\ell_t:\Theta^h \rightarrow \RR$ is an $(\eps, h)$-joint PEO if 
$|\ell_t(\theta_{1:h}^{1:k}) - c_t(\tilde{x}_t, \tilde{u}_t)| \le \eps$
where $(\tilde{x}_{1:t-h}, \tilde{y}_{1:t-h}, \tilde{u}_{1:t-h-1}) = (x_{1:t-h}, y_{1:t-h}, u_{1:t-h-1})$  and for $s> t-h$
\begin{align}
\tilde{u}_{s} &= \pi_{\theta^{1:k}_{t-s+1}}(\tilde{y}_{1:s}, \tilde{u}_{1:s-1}),\label{eq:joint_eval_rollout}\\
    \tilde{x}_{s+1} &= f_s(\tilde{x}_t, \tilde{u}_s) + w_s, \quad 
    \forall i \in [k], \tilde{y}^i_{s+1} = g_i(\tilde{x}_{s+1}) + e^i_{s+1} \nonumber~. 
\end{align}

$\ell^i_t:\Theta^h_i \rightarrow \RR$ is an $(\eps, h)$-local PEO for agent $i$ if 
$|\ell^i_t(\theta_{1:h}) - c_t(\tilde{x}_t, \tilde{u}_t)| \le \eps$
where $(\tilde{x}_{1:t-h}, \tilde{y}^i_{1:t-h}, \tilde{u}_{1:t-h-1}) = (x_{1:t-h}, y^i_{1:t-h}, u_{1:t-h-1})$, $\tilde{u}^{-i}_{t-h+1:t} = u^{-i}_{t-h+1:t}$  and for $s> t-h$
\begin{align}
\tilde{u}^{i}_{s} &= \pi_{\theta^{i}_{t-s+1}}(\tilde{y}^i_{1:s}, \tilde{u}_{1:s-1}), \label{eq:eval_rollout}\\
    \tilde{x}_{s+1} &= f_s(\tilde{x}_t, \tilde{u}_s) + w_s, \quad 
    \tilde{y}^i_{s+1} = g_i(\tilde{x}_{s+1}) + e^i_{s+1} \nonumber~. 
\end{align}
\end{definition}
\section{More Applications for Linear Dynamical Systems}\label{sec:LDS_applications}
\paragraph{Partially Observed LDS with Disturbance Response Control}
Another setting that we can capture with out meta-algorithm is partially observed linear dynamical systems, where all agents share the same partial observation. In this case, the dynamics are as follows,
\begin{align*}
    y_{t} &= g(x_{t}) + e_t = Cx_{t} + e_{t},\\
    x_{t+1} &= f_t(x_t,u_t) + w_t = A x_t + B u_t + w_t~,
\end{align*}
with bounded adversarial disturbances $e_t, w_t$ and known strongly stable dynamics.  For a partially observed system with adversarial disturbances, it is known that many states can be consistent with the observations and so costs must be restricted to be well defined for the agents.  This is solved by using convex costs of the \emph{observation} $c_t: \RR^{d_y} \times \RR^{d_u}$, so at time $t$ the cost incurred is $c_t(y_t, u_t)$.  Because $y_t$ is an affine function of $x_t$, this cost is still convex in $x$ to fit the general setting.

A policy class, that suits our needs is that of Disturbance-Response Control (DRC). The satisfaction of the assumptions of Algorithm~\ref{algorithm:multiagent_ctrl_generic} follows from similar arguments to DAC in the fully observable case. In particular, to compute PEO $\ell^i_t$, the following approximate Markov operator is used
\begin{align} \label{eq:partial_markove}
    y_t &\approx y^{\text{nat}}_t + \sum_{i=1}^k \hat{G} u^{i}_{t-1:t-h}\\
        \hat{G}&=[CB, CAB, \dots CA^{h-1} B]~. \nonumber
\end{align}

In the above equation $\hat{G}_i$ is the block of the Markov operator $\hat{G}$ corresponding to the $i$th agents controls. Beyond this, the core ideas are the same, just replacing disturbances with Nature's $y$'s. The last piece necessary is computing Nature's $y$'s.  This is not challenging given we have access to all agent's controls, so we can let $y^{\text{nat}}_{t} \approx y_t - \hat{G}_i u^{i}_{t-1:t-h}$.  For a complete analysis of this approach in the single agent setting, refer to \cite{simchowitz2020improper}.

\paragraph{Locally Partially Observed LDS with Disturbance Response Control}
An extension of the prior setting of interest is for each agent to have it's own local partial observation without needing to share it in the protocol.  In particular, consider the multi-agent system
\begin{align*}
    x_{t+1} &= f_t(x_t,u_t, w_t) = A x_t + B u_t + w_t,\\
    y^{i}_{t+1} &= g_i(x_{t+1}, e_t) = C^ix_{t+1} + e^{i}_{t+1}~,
\end{align*}

with bounded adversarial disturbances $e^{i}_t, w_t$ and known strongly stable dynamics. As noted in Remark~\ref{rmk:multi_pobs}, in order for each agent to have a local PEO, the cost $c_t(x_t, u_t)$ needs to be computable with only local information at each agent. Therefore, we assume there are local cost functions $c^i_t:\RR^{d_{y_i}} \times \RR^{d_u}$ for each agent such that
for all $i \in [k], c_t(x, u) = c^i_t(C^i x + e^i_t, u_t)$.

With access to the global loss in each agent, agent $i$ can use DRC with $y^{\text{nat}, i}_t$ to satisfy all assumptions required for our meta-algorithm.  Each observation matrix $C^i$ produces a different markov operator $G^{i} = [C^iB, \dots , C^iA^{h-1}B]$, such that
$y^{i}_t = y^{\text{nat},i}_t + G^{i} u_{t-1:t-h}$.
Using $G^i$ agent $i$ can compute $y^{\text{nat},i}_t$ and build a PEO in the same way as the single partial observation setting.

\paragraph{Linear Dynamics with Neural Network Policies}
All the previous policy parameterizations were linear in some feature representation. Recent work shows that nonstochastic control can extend beyond linear policies using analysis from deep learning theory \cite{chen2021provable}. The authors consider the fully observed setting with fully-connected ReLU neural network policies that act on normalized windows of disturbances. A regret bound for online episodic control is shown by proving that under some technical conditions, a PEO is $\eps$-\emph{nearly-convex}, which means 
\begin{align*}
    p(x) - p(y) \le \ip{\nabla p(x)}{x-y} - \eps~.
\end{align*}

It can be readily shown that Theorem~\ref{thm:multi_ocomem} can be extended to nearly convex functions using near-convexity in place of convexity, so we can get a multi-agent control regret guarantee for two-layer neural nets that act on normalized windows of disturbances. In addition, the proof that a PEO is nearly convex can be extended to neural net policies that act on normalized windows of Nature's $y$'s, so neural net policies can also be used in the partial observation settings described above.
\section{Robustness to System Failures: Overactuated Aircraft}\label{sec:admire}

In this section we provide the mathematical model of an overactuated aircraft based on ADMIRE \cite{harkegaard2005resolving}, and demonstrate the robustness of our controller to system failures.

Let $\alpha,\beta,p,q,r$ be angle of attack, slideslip angle, roll rate, pitch rate and yaw rate respectively. State space can be given as $x=\begin{bmatrix} \alpha & \beta & p & q & r
\end{bmatrix}^T$ and input space consists of 4 one-dimensional controllers, $\begin{bmatrix} u^1 & u^2 & u^3 & u^4
\end{bmatrix}.$ Discrete linear dynamics of the system can be given as $x(t+1)= Ax(t)+\sum_{i=1}^4B_iu^i(t)+W(t)$ where,
\begin{align*}
A=\begin{bmatrix}
1.5109 & 0.0084 & 0.0009 & 0.8598 & -0.0043 \\
0 & -0.0295 & 0.0903 & 0 & -0.4500 \\
0 & -3.1070 & -0.1427 & 0 & 2.7006\\
2.3057 & 0.0097 & 0.0006 & 1.5439 & -0.0029\\
0 & 0.5000 & 0.0125 & 0 & 0.4878
\end{bmatrix}\\ B=\begin{bmatrix}B_1 & B_2 & B_3 & B_4\end{bmatrix}=\begin{bmatrix}
0.6981 & -0.5388 & -0.5367 & 0.0029 \\
0 & -0.2031 & 0.2031 & 0.3912 \\
0 & -2.0768 & 2.0768 & -0.4667\\
1.8415 & -1.4190 & -1.4190 & 0.0035\\
0 & -0.1854 & 0.1854 & -0.7047
\end{bmatrix}
\end{align*}
and $W(t)$ is the disturbance. 

\subsection{Experimental Setup}

Controller details:
\begin{enumerate}
    \item GPC has a decaying learning rate of $\frac{0.001}{t}$, a rollout length $h$ of $5$ and uses DACs with windows of size $5$.
    \item MAGPC is split into $4$ $1$-d controllers, each using decaying learning rate of $\frac{0.001}{t}$, a rollout length $5$ and a DAC with window length $5$.
    \item We use a linear $\cH_{\infty}$ controller and standard infinite horizon LQR linear controller.
\end{enumerate}

Disturbance details:
\begin{enumerate}
    \item Random walk chooses $w_t = w_{t-1} + X_t$ where $X_t$ is a standard Gaussian random variable.
    \item Gaussian noise is iid. and has variance $1$.
    \item Sinusoidal noise is chosen via $w_i= \sin(2t + \phi_i)$ where $\phi = [12.0, 21.0, 3.0, 42.0, 1.0]$, picked arbitrarily once.
\end{enumerate}

Fourth control is set to $0$ in one experiment and otherwise controls are left unchanged.

\end{document}